\documentclass[11pt, twoside, a4paper]{article}
	\usepackage{amsmath,amssymb,amsthm,amscd,appendix,calrsfs,cite,color,epsfig,eucal,enumitem,dsfont,graphics,graphicx,latexsym,mathrsfs,verbatim,hyperref}
	\newtheorem{theorem}{Theorem}
	\newtheorem{corollary}{Corollary}
	\newtheorem{lemma}{Lemma}
	
	\newtheorem{remark}{Remark}
	
	\newtheorem{proposition}[theorem]{Proposition}

	\newcommand{\keywords}[1]{\par\addvspace\baselineskip\noindent\textbf{Keywords:}\enspace\ignorespaces#1}
	
	\newcommand{\AMSclassification}[1]{\par\addvspace\baselineskip\noindent\textbf{Mathematical subject classification:}\enspace\ignorespaces#1}
	
	\newcommand{\acknowledgment}[1]{\par\addvspace\baselineskip\noindent\textbf{Acknowledgment:}\enspace\ignorespaces#1}
	
	\DeclareMathOperator{\Supp}{Supp} 	
	\DeclareMathOperator{\Id}{Id}

	\title{Dynamical obstruction \\ to the existence of continuous sub-actions \\ for interval maps with regularly varying property}
	\author{Eduardo Garibaldi\thanks{Supported by CNPq grant 304792/2017-9.} \\ \small{Department of Mathematics, University of Campinas, 13083-859 Campinas, Brazil} \\ \small{(email: garibaldi@ime.unicamp.br)} \\ ~ \\
		Irene Inoquio-Renteria\thanks{Supported by FONDECYT 11130341 and BCH-CONICYT postdoctoral fellowship 74170014.} \\ \small{ICFM, Universidad Austral de Chile, casilla 567 Valdivia, Chile} \\
		\small{(email: ireneinoquio@uach.cl)}}
	
\begin{document}
		\maketitle
		\begin{abstract}
			In ergodic optimization theory, the
			existence of sub-actions is an important tool in the study of the so-called optimizing measures. For  transformations with regularly varying property,
			we highlight a class of moduli of continuity which is not compatible with  the existence of continuous sub-actions.  Our result relies fundamentally
			on the local behavior of the dynamics near a fixed point and  applies  to interval  maps that are expanding outside an neutral fixed point, including
			Manneville-Pomeau and Farey maps.
			
			\keywords{ergodic optimization, sub-actions, modulus of continuity, non-uniformly expanding dynamics, intermittent maps.}
			
			\AMSclassification{26A12, 37E05, 49L25, 49N15.}
		\end{abstract}

		\section{Introduction}\label{S:S1}
		
		Let $T:X \to X$ be a continuous surjective map on a compact metric space $X$.
		Suppose that $f:X\to \mathbb R$ is a continuous function (called \emph{potential}).
		Let $ M(X,T)$ denote the set of $T$-invariant Borel probability measures on $X$.
		As usual the maximum ergodic average is defined as
		\begin{equation*}
		m(f, T):= \max_{\mu \in M(X,T)} \int f \; d\mu.
		\end{equation*}
		Given a potential $f:X \to \mathbb R$, a function $u:X\to \mathbb R$ is said to be a
		\emph{sub-action} for $f$ if it satisfies the cohomological inequality
		\begin{equation*}
		f + u - u\circ T \le m(f,T).		
		\end{equation*}
		The existence of sub-actions for a potential  $f$ plays an important role in the study of measures $\mu$ in~$M(X,T)$ that maximize (or minimize) the average $\int_X fd\mu$.
		The study of these measures gave rise to the ergodic optimization (see~\cite{Jen06,Jen18,Gar17} and references therein).
		
		The existence of continuous sub-actions is guaranteed when the map is uniformly expanding and the potentials have H\"older modulus of continuity
		(see~\cite{CLT01} for the context  of expanding transformations of the circle). For related studies on the existence of sub-actions, see~\cite{LT03, LT05, LRR07, GLT09}, and
		see also \cite{Sou03, BraF07, Bra08, Mor09} for results in one-dimensional dynamics.
		
		For transitive expanding dynamics, generic continuous potentials do not admit bounded measurable sub-actions (see~\cite[Theorem~C]{BJ02} and for details~\cite[Appendix]{Gar17}).
		Surprisingly there are  few cases in the literature about specific examples of non-existence of continuous sub-actions. An example is provided by  Morris~\cite[Proposition 2]{Mor07} in the context of shift spaces.
		
		Our theorem highlights a dynamical obstruction on the existence of continuous  sub-actions.
		It seems that Morris~\cite{Mor09} was the first to notice this kind of phenomenon.
		Although our result holds for interval dynamics, we are convinced that such an obstruction must occur in a similar way for multidimensional settings.
		Precisely, we deal with interval maps with a regularly varying property and we identify an associated class of moduli of continuity
		whose members  do not always admit continuous sub-actions.  We present our theorem in the following subsection.
		In the Appendix, we address the natural question of the existence of sub-actions in such a setting.
		
		\subsection{Statement of the result}\label{statements}
		Let $[0,1]$ be endowed with the standard metric on $\mathbb R$. Our dynamical setting will be interval maps~$T: [0,1]\to [0,1],$  defined  for $x$ close enough to $0$ as an invertible function  of the form $T(x):=x(1\pm V(x))$, where for  some $\sigma>0$,  the continuous and increasing function $V:[0,+\infty)\to (0,1)$ satisfies
		\begin{equation}\label{regularly variation}
		\lim_{x\to 0} \frac{V(tx)}{V(x)} = t^{\sigma}, \textrm{ for all } t>0.
		\end{equation}
		The function $V$ is said to be \emph{regularly varying at} $0$ \emph{with index} $\sigma$.
		
		By a \emph{modulus of continuity}, we mean a continuous and non-decreasing function $\omega:[0,+\infty)\to[0,+\infty)$ satisfying~$ \lim_{\epsilon \to 0} \; \omega(\epsilon) = \omega(0) = 0 $.
		Let $\mathcal M$  denote  the family of concave modulus of continuity. For a given $\omega\in \mathcal M$,  we denote by $\mathscr{C}^{\omega}([0,1])$ the space of
		functions~$\varphi:[0,1]\to~\mathbb R$  with a multiple of $\omega$ as modulus of continuity:
		$ |\varphi(x)-\varphi(y)|\le C \omega(d(x,y)) $ for some constant $ C > 0 $, for all $ x,y\in [0,1] $.	
		
		\begin{theorem}\label{t:Theorem 1}
			Let $T:[0,1]\to [0,1]$ be an interval map such that, for $x$ close to $0$, $T$ is invertible and has the form $T(x):= x(1\pm V(x))$,  where the continuous and increasing function
			$V:[0,+\infty)\to (0,1)$ is regularly varying at $ 0 $ with index $\sigma>0$.
			Suppose that $\omega\in \mathcal M$ satisfies
			\begin{equation}\label{positive liminf}
			\liminf_{x\to 0}\frac{\omega(x)}{V(x)}>0.
			\end{equation}
			Then there exists a function $f\in \mathscr{C}^{\omega}([0,1])$,  with $ m(f,T) = \int f \, d\delta_0 = f(0) $, that does not admit continuous sub-action.
		\end{theorem}
		
		The main novelty of our general result is the clear identification of condition~\eqref{positive liminf} as an obstruction to the existence of continuous sub-actions.
		An immediate question is whether the opposite condition, that is, a null limit inferior would be sufficient to ensure existence.
		As a complement of discussion as well as an initial answer, we provide an example of existence result in the Appendix (see Theorem~\ref{theorem existence}):
		by considering certain maps with an indifferent fixed point and a stronger assumption than a null limit inferior,
		we show that sub-actions do exist and we highlight their associated regularity.

		In the following Subsection, we give examples of applications of Theorem~\ref{t:Theorem 1}.
		We gather in Section~\ref{Preliminares} preliminary results.
		In Section~\ref{Section proof of Theorem}, we present the proof of Theorem~\ref{t:Theorem 1}.
		In the Appendix, we detail the context and the proof of an existence result inspired  by a construction due to Contreras, Lopes and Thieullen \cite[Proposition 11]{CLT01}.
		
		\subsection{Examples}
		
		A trivial example of elements of $\mathcal M$ are the functions $\omega(h)=C h^\alpha$ with $\alpha\in(0,1]$, which describe $\alpha$-H\"older continuous functions.
		The family $\mathcal M$ also includes the minimal concave majorants $\omega_0$ of non-decreasing subadditive functions $\omega:[0,+\infty)\to~[0,+\infty)$, with $ \lim_{h \to 0} \omega(h)=\omega(0)=0.$
		Following~\cite{Med01} these concave majorants are infinitely differentiable on $(0,+\infty)$. Moreover, if $\omega'(0)<\infty$ then $\omega_0(h)=\omega'(0) h$ on some neighborhood of $0$.
		
		Another example of members of $\mathcal M$ are the functions $\omega(h)=h\left(\log\left(\frac{1}{h^k}\right)+1\right)$ (for $ k > 0 $ and $ h $ small enough),
		which describe locally H\"older continuous functions.
		A more general class of modulus of continuity in $\mathcal M$ is defined as follows: for  $0\le\alpha<1$ and $\beta\ge 0$ with $ \alpha + \beta > 0 $, consider $\omega_{\alpha,\beta}:[0,+\infty)\to [0,+\infty)$ given as
		\begin{equation}\label{modulus alpha beta}	
		\omega_{\alpha,\beta}(h):=\left\{
		\begin{array}{ll}
		h^\alpha(-\log h)^{-\beta},&0< h< h_0,  \\
		h_0^\alpha(-\log h_0)^{-\beta},& h\ge h_0,
		\end{array}
		\right.
		\end{equation}
		where $h_0$ is taken small enough so that $\omega_{\alpha,\beta}$ is concave.
		Note that $\omega_{\alpha,0}$ is reduced to the H\"older continuity, and $\omega_{0,\beta}$ for $\beta>0$ determines a class that is 	
		larger than local H\"older continuity -- see property~\eqref{quotient to infinite}.	
		
		\begin{remark}
			Let~$\omega_{\alpha,\beta}:[0,+\infty)\to [0,+\infty)$ be the modulus of continuity defined in~(\ref{modulus alpha beta}).	It is easy to see that
			for every $\epsilon>\alpha$,
			\begin{equation}\label{quotient to infinite}
			\displaystyle \lim_{h\to 0}\frac{\omega_{\alpha,\beta}(h)}{h^{\epsilon}}= +\infty.
			\end{equation}
		\end{remark}
		
		Note that  $\mathcal M$ includes many functions besides the previous examples for the simple fact that for each pair  $\omega_1, \omega_2\in \mathcal M$,
		we have  $\omega_1\circ\omega_2\in \mathcal M.$
		However, we are interested in a class of modulus of continuity whose behavior near $0$ satisfies condition~(\ref{positive liminf}), which is dictated by the dynamics.
		
		\medskip
		Let $V:[0,+\infty)\to (0,1)$  be a continuous and increasing function  which is regularly varying at $ 0 $ with index $\sigma>0$.
		Consider the modulus of continuity $\omega_{\alpha,\beta}$ defined in~(\ref{modulus alpha beta}) with $0 \le \alpha < \min\{\sigma, 1\} $ and $\beta \ge 0$  such that  $\alpha+\beta>0$.
		Thanks to property~(\ref{quotient to infinite}), the condition $\liminf_{x\to 0}\frac{\omega_{\alpha,\beta}(x)}{V(x)}>0$ holds whenever
		$\liminf_{x\to 0}\frac{x^\sigma}{V(x)}>~0$.  Therefore, we obtain the following corollary.
		
		\begin{corollary}\label{Corollary 1}
			Let $T:[0,1]\to [0,1]$ be an interval map such that in a neighborhood of the origin
			$T$ is invertible and has the form $T(x) = x(1\pm V(x))$,
			where $V:~[0,+\infty)~\to~(0,1)$ is a continuous, increasing and regularly varying function at $ 0 $ with index $\sigma>0$	
			that satisfies $\liminf_{x\to 0}\frac{x^{\sigma}}{V(x)}>0.$ Let $\omega_{\alpha,\beta}(x)$ be defined as in~(\ref{modulus alpha beta}). Then, for $ \alpha = \sigma $ and $ \beta = 0 $ or
			for $ 0 \le \alpha < \min\{\sigma, 1\} $ and $\beta \ge 0$  with  $\alpha+\beta>0$,  there is a function~$f\in\mathscr{C}^{\omega_{\alpha,\beta}}([0,1])$ which does not admit continuous sub-action.
		\end{corollary}

		Examples of this kind of dynamics include Manneville-Pomeau interval map:
		for a given $s>0$, $ T_s:[0,1]\to [0,1]$ is defined as
		$$  T_s(x):= x(1 + x^s) \mod 1. $$
		Note that $ T'_s(x)\ge 1$ for all $ x $ with equality only at $x=0$.
		Let $c $ be the unique point in $ (0,1)$ such that $ T_s(c)=1$ and
		$ T_s\vert_{[0,c]}:[0,c]\to [0,1]$ is a diffeomorphism.
		Let us denote $ U_s:[0,1]\to [0,c]$ the corresponding inverse branch. Note that
		$U_s'(x)\le 1$ for  all $x$ and $U_s$ is concave, so that
		$cx \le U_s(x)\le x$. 	
		If we write $U_s(x)= x(1-V(x))$, then
		$	0\leq V(x)\le 1-c. $
		Moreover, by using  the identity $T_s\circ U_s=\Id$, we have
		$V(x)= x^s(1-V(x))^{s+1}$ for all $x\neq 0.$ Hence  $\lim_{x\to 0}V(x)=0$,
		\begin{align*}
		\lim_{x\to 0}\frac{V(tx)}{V(x)}=\lim_{x\to 0}t^s\left(\frac{1-V(tx)}{1-V(x)}\right)^{s+1}= t^s
		\textrm{ and }
		\lim_{x\to 0}\frac{x^s}{V(x)}=\lim_{x\to 0}\frac{1}{(1-V(x))^{s+1}}=1.
		\end{align*}
		It is not difficult to argue that $ V $ is increasing. Then Corollary~\ref{Corollary 1} applies to $U_s$ as well.

		\begin{corollary}
			Let $s\in (0,1)$ and  $ T_s(x)= x+x^{1+s} $ for $x$ close enough to $0$.  Denote $U_s$ the corresponding inverse branch. Let $\omega_{\alpha,\beta}(x)$ be defined as  in~(\ref{modulus alpha beta}), where either  $\alpha\in [0,\min\{s,1\})$ and $\beta\ge 0$ with   $\alpha+\beta> 0$ or $\alpha=s$ and $\beta=0$. Then there are functions $f,g \in \mathscr{C}^{\omega_{\alpha,\beta}}([0,1])$  which do  not admit continuous sub-actions with respect to $T_s$ and $U_s$, respectively.
		\end{corollary}
		The above corollary is an extension of Morris' result~\cite{Mor09}, which established  that for  $ T_s(x)= x+x^{1+s} \mod 1$, there is $f\in \mathscr C^{\omega_{s,0}}([0,1])$ that does not admit continuous sub-action.

		Another one-parameter family of maps on the interval $[0,1]$ with indifferent fixed point at $x=0$ is defined as follows: for  $\rho\in (0,1]$, let
		$ F_\rho:[0,1]\to [0,1]$ be given as
		\begin{equation*}
		F_\rho(x) =\left\{
		\begin{array}{ll}
		\frac{x}{(1-x^{\rho})^{1/\rho}}  & \textrm{ if } 0\le x\le 2^{-1/\rho}\\
		\frac{(1-x^{\rho})^{1/\rho}}{x} &\textrm{ if }  2^{-1/\rho}<x\le 1.
		\end{array}
		\right.
		\end{equation*}
		Note that Farey map corresponds to the special case $\rho=1$. For any $\rho\in (0,1]$, the first inverse branch has an explicit expression: $ G_\rho(x) = \frac{x}{(1+x^{\rho})^{1/\rho}} $.
		Note then that the functions  $V(x)= \frac{1}{(1-x^\rho)^{1/\rho}}-1$ and $W(x)= 1-\frac{1}{(1+x^\rho)^{1/\rho}}$ are continuous, increasing, regularly varying with index~$\rho$, and satisfy
		$\lim_{x\to 0}\frac{x^\rho}{V(x)}= \lim_{x\to 0}\frac{x^\rho}{W(x)}=\rho>0.$ Clearly, $ F_\rho(x) = x(1 + V(x)) $ and $ G_\rho(x) = x(1 - W(x)) $.

		\begin{corollary}
			For $\rho\in(0,1]$, let $F_\rho(x) = \frac{x}{(1-x^{\rho})^{1/\rho}} $ and  $G_\rho(x)=\frac{x}{(1+x^{\rho})^{1/\rho}}$ for $x$ close to $0$. Let $\omega_{\alpha,\beta}(x)$ be defined as in~(\ref{modulus alpha beta}), where either $\alpha\in [0,\rho)$ and $\beta\ge 0$ with   $\alpha+\beta> 0$ or $\alpha=\rho$ and $\beta=0$. Then there are functions $f, g\in \mathscr{C}^{\omega_{\alpha,\beta}}([0,1])$   which do not admit continuous sub-actions with respect to $F_\rho$ and $G_\rho$, respectively.
		\end{corollary}
		
		As a final example of application of our theorem, let
		\begin{equation*}
		T(x) =\left\{
		\begin{array}{ll}
		0 & \textrm{ if } x = 0 \\
		x+ \frac{2}{\log 2} \, x^2 \vert \log x\vert  & \textrm{ if } 0 < x\le 1/2 \\
		2x-1 &\textrm{ if } 1/2 < x \le 1.
		\end{array}
		\right.
		\end{equation*}
		Note that $V(x)=\frac{2}{\log 2} \, x\vert \log x\vert$, $ x > 0 $, is a regularly varying function with index~$1$. For $k>0$, the concave modulus of continuity
		defined for $ h $ sufficiently small as $\omega(h)=h\left(\log\left(\frac{1}{h^k}\right)+1\right)$
		clearly satisfies $\lim_{x\to 0}\frac{\omega(x)}{V(x)}=\frac{2k}{\log 2}>0$. Recalling that such a modulus describes locally H\"older continuous functions, we have the following result.
		\begin{corollary}
			With respect to a dynamics that behaves as $T(x)= x+ \frac{2}{\log 2} \, x^2 \vert \log x\vert $ for $ x > 0 $ sufficiently small, there exist locally H\"older continuous functions that do not admit continuous sub-actions.
		\end{corollary}

		\section{Preliminaries}\label{Preliminares}
		
		\begin{subsection}{Some facts about modulus of continuity}
			
			Recall that $\mathcal M$ denotes the family of concave modulus of continuity.
			Note that, given a  non-identically null $\omega \in \mathcal M$, then $([0,1],\omega\circ d)$ is a metric space. Indeed, the subadditivity  of $\omega$ follows from its concavity and thus, since $\omega$ is non-decreasing,
			we obtain the triangle inequality:		
			\begin{equation*}
			\omega(d(x,y))\le \omega(d(x,z))+\omega(d(z,y))\quad  \forall \, x,y,z\in [0,1].
			\end{equation*}
			In particular, a function $\varphi:[0,1]\to \mathbb R$ with modulus of continuity  $\omega\in \mathcal M$ is nothing else than a Lipschitz function with respect to the metric $\omega\circ d.$
			
			We will use the following property.
			\begin{lemma}\label{modulo constante}
				Let $\omega\in \mathcal M$. For any positive constant $\chi,$ we have
				\begin{equation*}
				\frac{\chi}{1+\chi} \omega(h) \le \omega(\chi h)\le (\chi+1)\omega(h).
				\end{equation*}
			\end{lemma}
			\begin{proof}
				Since $\omega $ is subadditive, we have for all positive integer $n\ge 1$,  $\omega(n h)\le~n\omega(h)$. For a positive constant $\chi$, by monotonicity of $\omega$, we see that
				$$\omega(\chi h)\le \omega(\lceil \chi \rceil h)\le \lceil \chi \rceil\omega(h)\le (\chi+1)\omega(h),$$
				where $ \lceil \cdot \rceil $ denotes the ceiling function.
				Then, we also obtain
				$$\omega(\chi h)\ge \frac{1}{\frac{1}{\chi}+1}\omega(h)= \frac{\chi}{1 + \chi}\omega(h).$$
			\end{proof}
		\end{subsection}
		\subsection{Local behavior  near a fixed point}\label{S:S3}
		
		Given $\sigma>0$, a measurable function $V: [0,+\infty)\to (0,+\infty)$  is said to be  \emph{regularly varying at $0$ with index} $\sigma$ if condition~\eqref{regularly variation} holds.
		A regularly varying function can be represented in the form $V(x)= x^{\sigma}\mathcal V(x) $, where the function~$\mathcal V $ satisfies
		$ \lim_{x\to 0} \frac{\mathcal V(tx)}{\mathcal V(x)} =1 $, for all $ t > 0 $.
		Similarly a measurable function  $V:[0,+\infty)\to (0,+\infty)$ is \emph{regularly varying at $\infty$ with index}~$\sigma\in \mathbb R$ if the function $x \mapsto V(\frac{1}{x})$ is regularly varying at $0.$
		For properties of regularly varying functions, we refer to~\cite{Sen76} and \cite{Aar97}. See also~\cite{Kar33} for details concerning the original literature.
		
		Recall that near to origin the dynamics is supposed invertible and defined as $T(x)=x(1\pm V(x))$.
		Let  $(w_n)_{n=0}^{+\infty} \subset [0,1]$ be a sequence of points obtained by choosing  $w_0$ close enough to $0$ and by defining  $w_{n+1}= T^{\mp 1}(w_n)$, $ n \ge 0$. In clear terms, for $x\mapsto  x(1+V(x))$ we take pre-images, and for $x\mapsto x(1-V(x))$ we consider future iterates. Note that in both cases  $w_n\to 0$ as $n\to \infty.$ A sequence of iteration times will also play  a central role in our construction. More precisely, let $(n_k)_{k\ge 1}$ be an increasing sequence of positive integers such that for some $\gamma\in (0,1)$,
		\begin{equation}\label{definition of gamma}
		\lim_{k\to \infty} \frac{n_k}{n_{k+1}}= \gamma.
		\end{equation}
		The study of the behavior close to $0$ can be done in a similar way for both  $x\mapsto x(1+V(x))$ and  $x\mapsto x(1-V(x))$. From now on in this subsection, we look  at the case $T(x)= x(1-V(x))$. We will point out in the end similarities and particularities  to the other case

		We write $ \alpha_j \sim \beta_j $ whenever $ \frac{\alpha_j}{\beta_j} \to 1$ as $j\to \infty$. The next lemma summarizes the main properties concerning the asymptotic behavior of the
		sequences~$\left(w_n=T(w_{n-1})\right) $ and $(n_k)$.
		
		\begin{lemma}\label{lema assintoticos} The following properties hold
			\begin{enumerate}
				\item[(i)]
				\begin{equation}\label{asymptotic}
				w_n\sim \displaystyle\frac{1}{\sigma^{1/\sigma}b(n)}, \quad \textrm{ where } b^{-1}(x):=\frac{1}{V(\frac{1}{x})};
				\end{equation}	
				\item[(ii)]
				\begin{equation}\label{sequencias consecutivas}
				d(w_n,w_{n+1})\sim \frac{1}{\sigma^{1+1/\sigma}}\frac{1}{n b(n)};
				\end{equation}
				\item[(iii)]
				\begin{equation}\label{gamma mais 1}
				\frac{n_k}{n_{k+1}} \sim \gamma^{1+ 1/\sigma} \frac{b(n_{k+1})}{b(n_k)} .  	
				\end{equation}
			\end{enumerate}
		\end{lemma}
		\begin{proof}
			To verify Part~(iii), we first note that $ \frac{b^{-1}(tx)}{b^{-1}(x)} = \frac{V(1/x)}{V(1/tx)} \to \frac{1}{(1/t)^\sigma} = t^\sigma $ as $ x \to \infty $, which means that $ b^{-1} $ is regularly varying at $\infty$ with  index $\sigma$.
			Hence, its inverse, the increasing function $ b $, is regularly varying at  $\infty$ with index $1/\sigma$ (for details, see \cite{Sen76}).
			
			We set $b(y)= y^{1/\sigma} \mathcal B(y)$, where $\lim_{y\to \infty}\frac{\mathcal B(ty)}{\mathcal B(y)}=1$, for every $t>0.$ The function $\mathcal B$ has the following representation
			(for a proof, see \cite[Theorem~$1.2$]{Sen76}): there exist  $Y>0$ and measurable functions $\Theta: [Y,\infty)\to \mathbb R$, $\varepsilon: [Y,\infty)\to \mathbb (-\frac{\sigma}{2}, \frac{\sigma}{2}) $,  with
			$\Theta(y)\to \theta\in \mathbb R^+$ as $y\to \infty$ and $\varepsilon(t)\to 0$ as $t\to \infty$, such that
			\begin{equation*}
			\mathcal B(y)= \Theta(y)e^{\int_{Y}^{y} \frac{\varepsilon(t)}{t}dt}  \qquad  \forall \, y \ge Y.
			\end{equation*}
			Then
			\begin{equation*}
			\log \frac{\mathcal B(n_k)}{\mathcal B(n_{k+1})} =
			\log\frac{\Theta(n_k)}{\Theta(n_{k+1})}+\int_{n_{k+1}}^{n_k} \frac{\varepsilon(t)}{t}dt \qquad \text{and}
			\end{equation*}
			\begin{equation*}
			\big(\sup_{[n_k,+\infty)}\varepsilon\big) \log \frac{n_k}{n_{k+1}}
			\le \int_{n_{k+1}}^{n_k} \frac{\varepsilon(t)}{t}dt
			\le \big(\inf_{[n_k,+\infty)} \varepsilon\big) \log \frac{n_k}{n_{k+1}}
			\end{equation*}
			ensure that  $\frac{\mathcal B(n_k)}{\mathcal B(n_{k+1})} \to 1$ as $k\to +\infty.$
			Therefore 	
			$$\frac{n_k b(n_k)}{n_{k+1} b(n_{k+1})}=\Big(\frac{n_k}{n_{k+1}} \Big)^{1+ 1/\sigma}\frac{\mathcal B(n_k)}{\mathcal B(n_{k+1})}\to \gamma^{1+ 1/\sigma} \textrm{ as } k\to\infty.$$
			Part~(i) follows from~\cite[Lemma~4.8.6]{Aar97} which  is deduced using that
			\begin{equation}\label{asymptotic a}
			b^{-1}\Big(\frac{1}{w_{n}}\Big)\sim n\sigma.
			\end{equation}
			The asymptotic equivalence~\eqref{asymptotic a} implies that
			$ V(w_{n})=1 / b^{-1}\big(\frac{1}{w_{n}}\big) \sim \frac{1}{n\sigma}$, so it follows that $d(w_n,w_{n+1}) = w_n V(w_n) \sim \frac{1}{\sigma^{1+1/\sigma}}\frac{1}{n b(n)} $ and therefore Part~(ii) holds.
		\end{proof}
		
		\begin{remark}\label{Remark consecutive sequences}
			Since $ b $ is  a continuous and increasing function and since we consider the standard metric on $\mathbb R$,  by the asymptotic equivalence~(\ref{sequencias consecutivas}),
			there exists a constant $C_0>1$ such that  for every $i \le j$,
			\begin{equation}\label{minoracao e majoracao}
			(j-i) C_0^{-1}\frac{1}{\sigma^{1+1/\sigma}}\frac{1}{j\, b(j)} \le d(w_i, w_j)\le (j-i) C_0\frac{1}{\sigma^{1+1/\sigma}} \frac{1}{i\, b(i)}.
			\end{equation}
		\end{remark}

		The next lemma provides us estimates on the cardinality of future iterates that stay within suitable intervals.
		
		\begin{lemma}\label{minoracao cardinalidade}
			Let us consider $(w_{n_k})_{k=1}^{+\infty}$ a subsequence  of $(w_n)_{n=0}^{+\infty}$, where
			$(n_k)_{k\ge 1}$ is an increasing sequence  satisfying~(\ref{definition of gamma}) and $T^{n_{k}-n_{k-1}}(w_{n_{k-1}})= w_{n_k}$. 	For $k\ge 1$, denote
			\begin{equation*}
			R_k:= \frac{1}{3C_0^3}\frac{n_{k-1} b(n_{k-1})}{n_{k}b(n_{k})} d(w_{n_k}, w_{n_{k-1}}).
			\end{equation*}
			Then, for $z\in [w_{n_{k}} + R_k, w_{n_{k-1}}]$ and $ k $ large enough,
			\begin{multline*}
			\# \big\{0\le j< n_k-n_{k-1} : R_{k} \le d(T^j(z),  w_{n_{k}}) \le \frac{1}{3} d(w_{n_{k}},  w_{n_{k-1}}) \big\} \ge \\
			\ge C_1 n_{k-1} b(n_{k-1}) d(w_{n_{k}},w_{n_{k-1}}),
			\end{multline*}
			where $ C_1 : =  \frac{1}{4} (C_0^{-1}-C_{0}^{-2})\sigma^{1+1/\sigma} >0 $. In particular,  there is $C_2>0$ such that,  for $k$ sufficiently large, 		  	
			\begin{equation*}
			\#\big\{0\le j< n_k-n_{k-1}: R_k \le d(T^{j}(w_{n_{k-1}}),  w_{n_k}) \le\frac{1}{3} d(w_{n_k},w_{n_{k-1}}) \big\}
			\ge \frac{C_2}{V(w_{n_k})}.
			\end{equation*}
		\end{lemma}

		\begin{proof}
			Let $ \ell \ge 1 $ be such that $ w_{n_{k-1} + \ell} < z \le w_{n_{k-1} + (\ell-1)} $.
			Note that a nonnegative integer $ j $  such that
			\begin{equation}\label{primeira condicao suficiente minoracao}
			R_{k} \le d(w_{n_{k-1} + \ell + j}, w_{n_k}) \quad \text{and} \quad
			d(w_{n_{k-1} + (\ell - 1) + j}, w_{n_k}) \le \frac{1}{3} d(w_{n_{k}}, w_{n_{k-1}})
			\end{equation}
			belongs to $ \big\{j : R_{k} \le d(T^j(z), w_{n_k}) \le \frac{1}{3} d(w_{n_{k}}, w_{n_{k-1}}) \big\} $. Moreover, thanks to~\eqref{minoracao e majoracao}, any $ j \ge 0 $ such that
			\begin{multline}\label{segunda condicao suficiente minoracao}
			R_k \le (n_{k} - n_{k-1} - \ell - j) C_0^{-1} \frac{1}{\sigma^{1+1/\sigma}}\frac{1}{n_{k}b(n_{k})} \quad \text{and} \\ (n_{k} - n_{k-1} - (\ell - 1) - j) C_0 \frac{1}{\sigma^{1+1/\sigma}}\frac{1}{n_{k-1} b(n_{k-1})}\le
			\frac{1}{3} d(w_{n_k}, w_{n_{k-1}})
			\end{multline}
			satisfies~\eqref{primeira condicao suficiente minoracao}. Denoting $ \kappa := n_{k} - n_{k-1} - \ell $, there are exactly
			$$ \lfloor \kappa - C_0\sigma^{1+1/\sigma} n_{k}b(n_{k}) R_{k} \rfloor - \lceil \kappa + 1 - \frac{1}{3} C_0^{-1} \sigma^{1+1/\sigma} n_{k-1}b(n_{k-1})  d(w_{n_{k}}, w_{n_{k-1}}) \rceil + 1 $$
			nonnegative integers $j$ that fulfill~\eqref{segunda condicao suficiente minoracao}. Therefore, we have
			\begin{align*}
			\# \big\{j : R_{k} \le & \,\, d(T^j(z),  w_{n_k}) \le \frac{1}{3} d(w_{n_{k}}, w_{n_{k-1}}) \big\} \ge \\
			& \ge  \frac{1}{3} C_0^{-1} \sigma^{1+1/\sigma} n_{k-1}b(n_{k-1})  d(w_{n_{k}}, w_{n_{k-1}}) -  C_0 \sigma^{1+1/\sigma} n_{k}b(n_{k}) R_k -2 \\
			& = \frac{1}{3} (C_0^{-1} - C_0^{-2}) \sigma^{1+1/\sigma} n_{k-1} b(n_{k-1})  d(w_{n_{k}}, w_{n_{k-1}}) - 2.
			\end{align*}
			Note that, from Remark~\ref{Remark consecutive sequences} and Lemma~\ref{lema assintoticos}, as $ k \to \infty $
			\begin{equation*}
			\sigma^{1+1/\sigma} n_{k-1} b(n_{k-1}) d(w_{n_{k}}, w_{n_{k-1}}) \ge C_0^{-1} n_{k} \Big( 1 - \frac{n_{k-1}}{n_{k}} \Big)
			\frac{n_{k-1} b(n_{k-1})}{n_{k} b(n_{k})} \to \infty.
			\end{equation*}
			Hence, ignoring  at most finitely many initial terms of $ (n_k) $ if necessary, we obtain
			$$
			\# \big\{j : R_{k} \le d(T^j(z),  w_{n_{k}}) \le \frac{1}{3} d(w_{n_{k}}, w_{n_{k-1}}) \big\}
			\ge C_1 n_{k-1} b(n_{k-1}) d(w_{n_{k}},w_{n_{k-1}}).
			$$
			
			In particular, for $z = w_{n_{k-1}}$, from~\eqref{minoracao e majoracao} we have
			\begin{align*}
			d(w_{n_k},w_{n_{k-1}})^\sigma  \# \big\{ j: & \,\, R_k \le d(T^j(w_{n_{k-1}}), w_{n_{k}}) \le \frac{1}{3} d(w_{n_{k}},w_{n_{k-1}})\big\} \ge \\
			&\ge
			C_1 d(w_{n_{k}},w_{n_{k-1}})^{\sigma+1} n_{k-1} b(n_{k-1})\\
			&\ge
			C_1 \Big[(n_{k}-n_{k-1}) C_0^{-1}\frac{1}{\sigma^{1+1/\sigma}} \frac{1}{n_{k}b(n_{k})}\Big]^{\sigma+1}  n_{k-1} b(n_{k-1})\\
			&=
			\frac{C_1}{C_0^{\sigma + 1} \sigma^{(\sigma+1)^2/\sigma}}\Big(1-\frac{n_{k-1}}{n_{k}}\Big)^{\sigma+1}  \frac{n_{k-1}b(n_{k-1})}{n_{k}b(n_{k})}\frac{n_{k}}{b(n_{k})^\sigma}.
			\end{align*}
			Note now that, from~\eqref{asymptotic} and \eqref{asymptotic a},
			\begin{equation*}
			\frac{n}{b(n)^\sigma}\sim \sigma nw_n^\sigma\sim \frac{w_n^\sigma}{V(w_{n})}.
			\end{equation*}
			Denote thus $ C_1' :=  \frac{1}{2} \frac{C_1}{C_0^{\sigma + 1} \sigma^{(\sigma+1)^2/\sigma}} (1-\gamma)^{\sigma+1} \gamma^{1+1/\sigma} > 0 $.
			Following the previous estimate and the above asymptotic equivalence, from~\eqref{definition of gamma} and \eqref{gamma mais 1}, for $ k $ large enough,
			$$\#
			\big\{ j: R_k \le d(T^j(w_{n_{k-1}}), w_{n_k}) \le \frac{1}{3} d(w_{n_{k-1}},w_{n_k})
			\big\}\ge \frac{C_1'}{V(w_{n_k})} \frac{w_{n_k}^\sigma}{d(w_{n_k},w_{n_{k-1}})^\sigma}.
			$$
			Note now that,  from Remark~\ref{Remark consecutive sequences} and Lemma~\ref{lema assintoticos}, for $ k $ sufficiently large,
			\begin{equation*}
			d(w_{n_k},w_{n_{k-1}}) \le \left(1-\frac{n_{k-1}}{n_{k}}\right) C_0 \frac{1}{\sigma}  \frac{n_{k} b(n_k)}{n_{k-1} b(n_{k-1})} \frac{1}{\sigma^{1/\sigma} b(n_k)}  \le 2(1-\gamma) C_0\frac{1}{\sigma}\frac{1}{\gamma^{1+1/\sigma}}  w_{n_k}.
			\end{equation*}
			We obtain thus a constant $ C_1'' > 0 $ such that $  \frac{w_{n_k}^\sigma}{d(w_{n_k},w_{n_{k-1}})^\sigma} \ge C_1'' $
			whenever $ k $ is large enough, which completes the proof with
			$ C_2 := C_1' C_1'' $.
		\end{proof}
		
		\paragraph{Comments on local behavior near to origin for $x\mapsto x(1+V(x))$.}
		In this case, we deal with a sequence of past iterates $\left(w_n= T(w_{n+1})\right)$, where $T(x)=x(1+V(x))$ in a neighborhood of $0$. It is not a surprise that asymptotic equivalences are exactly the same as in the statement of Lemma~\ref{lema assintoticos}. One may show easily such a fact with minor adjustments in the proof and an appropriate version of \cite[Lemma 4.8.6]{Aar97}, which can be obtained repeating almost verbatim original arguments. The statement of Lemma~\ref{minoracao cardinalidade} for this case obviously requires contextual changes since the sequences are now related  by $T^{n_{k}-n_{k-1}}(w_{n_{k}})= w_{n_{k-1}}$. If one follows the same lines of proof, one will conclude that  for $z\in [w_{n_k}, w_{n_{k-1}}-R_k]$ and $k$ large enough,
		\begin{multline*}
		\# \big\{0\le j< n_k-n_{k-1} : R_{k} \le d(T^j(z),  w_{n_{k-1}}) \le \frac{1}{3} d(w_{n_{k}},  w_{n_{k-1}}) \big\} \ge \\
		\ge C_1 n_{k-1} b(n_{k-1}) d(w_{n_{k}},w_{n_{k-1}}),
		\end{multline*}
		and in particular for  $k$ sufficiently large, 		  	
		\begin{equation}\label{segundo minorante alternativo}
		\#\big\{0\le j< n_k-n_{k-1}: R_k \le d(T^{j}(w_{n_{k}}),  w_{n_{k-1}}) \le\frac{1}{3} d(w_{n_k},w_{n_{k-1}}) \big\}
		\ge \frac{C_2}{V(w_{n_k})}.
		\end{equation}

		\section{Proof of Theorem~\ref{t:Theorem 1}}\label{Section proof of Theorem}
		We will present in details the proof of Theorem~\ref{t:Theorem 1} when $T(x)=x(1-V(x))$ for $x$ close to $0$. In the end, we will comment on the small changes of arguments required to prove the theorem in the case $x\mapsto x(1+V(x))$. Hence, let~$(w_{n_k})_{k=1}^{+\infty}$ be a subsequence  of future iterates $(w_n = T^n(w_0))_{n=0}^{+\infty}$,
		where $ w_0 \in (0,1) $ is a point close enough to $ 0 $ and $(n_k)_{k\ge 1}$ is an increasing sequence such that
		$  \lim_{k\to +\infty}\frac{n_k}{n_{k+1}}=\gamma $ for some $ \gamma \in (0, 1) $.
		
		Define then
		\begin{equation*}
		S:=\{w_{n_k}\}_{k=1}^{+\infty}\cup \{0\}.
		\end{equation*}
		For every $k > 1$,  set
		\begin{align*}
		I_k & = \Big (\frac{1}{5} (3 w_{n_k} + 2 w_{n_{k+1}}), \frac{1}{5} (3 w_{n_k} + 2 w_{n_{k-1}}) \Big) \qquad \text{and} \\
		J_k & = \Big (\frac{1}{3} (w_{n_k} + 2 w_{n_{k+1}}), \frac{1}{3} (2 w_{n_k} + w_{n_{k+1}}) \Big),
		\end{align*}
		and denote $ Y:= (w_{n_1}, 1] \cup \bigcup_k J_k $. Since $ \{Y, \, I_k \,\, (k > 1)\} $ is an open cover  of $ ((0,1], \omega \circ d) $, we may consider a partition of unity subordinate to it (see Figure~\ref{Figure1}).
		
		Precisely, let  $\{ \varphi_Y, \, \varphi_k: ((0,1], \omega \circ d) \to [0,1] \,\, (k > 1)\} $ be a family of Lipschitz continuous functions such that  $ \varphi_Y + \sum_k \varphi_k=1$,
		with $\Supp(\varphi_Y)\subset Y$ and $\Supp(\varphi_k)\subset I_k $. In particular,  $ \omega $ is a modulus of continuity of $ \varphi_Y $ and of $ \varphi_k $ $ (k > 1) $.
		
		\bigskip
		
		\begin{figure}[h]
			\setlength{\abovecaptionskip}{30pt plus 3pt minus 2pt}
			\centering
			\includegraphics[height=1in,width=5in]{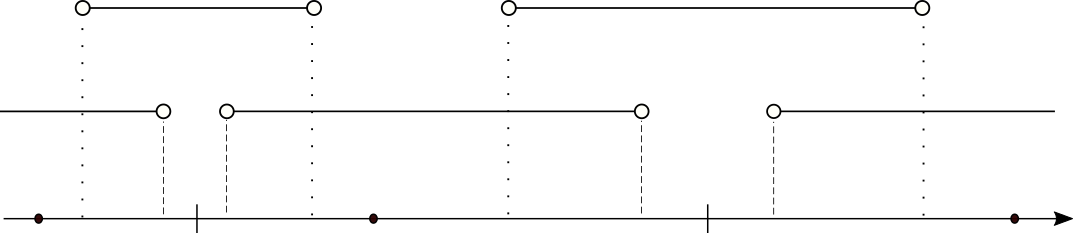}
			\caption{$d^{-}:=d(w_{n_k},w_{n_{k-1}})$, $d^{+}:=d(w_{n_k},w_{n_{k+1}})$}
			\label{Figure1}
			\setlength{\unitlength}{1cm}
			\begin{picture}(0,0)
			\thicklines
			\put(1.6,4.7){$J_{k-1}$}
			\put(-4.3,4.7){$J_{k}$}
			\put(-6.1,3.5){$I_{k+1}$}
			\put(-1.4,3.5){$I_{k}$}
			\put(4.6,3.5){$I_{k-1}$}
			\put(-2.1,1.8){$w_{n_k}$}
			\put(-6.2,1.8){$w_{n_{k+1}}$}
			\put(5.4,1.8){$w_{n_{k-1}}$}
			\put(-3.2,2.3){$2/5d^+$}
			\put(-0.6,2.3){$2/5d^-$}
			\put(-4.8,1.5){$\frac{w_{n_k}+w_{n_{k+1}}}{2}$}
			\put(1.2,1.5){$\frac{w_{n_{k-1}}+w_{n_k}}{2}$}
			\end{picture}
		\end{figure}
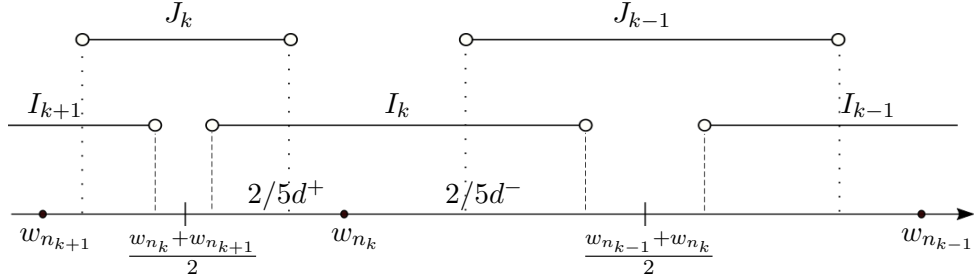
		For $\xi>0$, define $$\Phi(x):= \left \{
		\begin{array}{cc}
		\varphi_{k}(x),& x\in I_k, \,\, k=1 \mod 3 \\
		-\xi \varphi_{k}(x),& x\in I_k, \,\, k=2 \mod 3 \\
		0,&  \textrm{otherwise},
		\end{array}
		\right.$$ and consider $f: [0,1]\to \mathbb R$ given as
		\begin{equation}\label{definicao funcao}
		f(x):= \Phi(x)\, \omega(d(x,S)).
		\end{equation}
		This function clearly vanishes on $S$. Moreover,  $ f $ has $\omega$ as modulus of continuity.
		We will show that, for $ \xi $ large enough, $ f $ does not admit a continuous sub-action.
		
		We have $T^{m_k}(w_{n_{k-1}})= w_{n_{k}}$, where $m_k:= n_{k}-n_{k-1}$, and
		$$ S_{m_k} f(w_{n_{k-1}}) = \sum_{j= 0}^{m_k-1} f\left(T^j\left(w_{n_{k-1}}\right)\right)
		= \sum_{j=0}^{m_k-1} \Phi(w_{n_{k-1}+j})\, \omega(d(w_{n_{k-1}+j}, S)). $$
		Recall the definition of $ R_k $ in the statement of Lemma~\ref{minoracao cardinalidade}. Note that, for $ k $ large enough,
		$\left[w_{n_k}, w_{n_k}+R_k\right)\subset \left[w_{n_k}, \frac{1}{3}(2w_{n_k}+w_{n_{k-1}})\right)\subset I_k$. Besides, by construction $ \varphi_k \equiv 1 $ on
		$ \big[ \frac{1}{3}(2 w_{n_k} + w_{n_{k+1}}), \frac{1}{3} (2 w_{n_k} +  w_{n_{k-1}}) \big] $.
		Therefore, if $ k=1 \mod 3 $ is sufficiently large, from Lemma~\ref{minoracao cardinalidade} we get
		\begin{eqnarray*}
			S_{m_k} f(w_{n_{k-1}})
			&\ge & \#\big\{j:
			R_k \le d(w_{n_{k-1}+j}, w_{n_k}) \le \frac{1}{3} d(w_{n_k},w_{n_{k-1}})
			\big\}
			\omega(R_k)\\
			&\ge& \frac{C_2}{V(w_{n_k})} \omega(R_k).
		\end{eqnarray*}
		We will show that for $k$ sufficiently large,  $\frac{\omega(R_k)}{V(w_{n_k})}$ is  bounded from below by a positive constant. As a matter of fact, by the definition of $R_k$ and \eqref{gamma mais 1},
		$$  \lim_{k\to \infty}\frac{R_k}{d(w_{n_k},w_{n_{k-1}})}= \frac{1}{3} \frac{\gamma^{1+1/\sigma}}{C_0^3}. $$
		For $ C_3:=\frac{1}{4} \frac{\gamma^{1+1/\sigma}}{C_0^3} > 0 $, using the monotonicity of $\omega$ and Lemma~\ref{modulo constante}, we have that for a sufficiently large $k$,
		$$
		\omega(R_k)\ge \frac{C_3}{1+ C_3} \omega(d(w_{n_k},w_{n_{k-1}})).
		$$
		Moreover,  from Remark~\ref{Remark consecutive sequences} and Lemma~\ref{lema assintoticos}, we see that for $ k $ sufficiently large,
		\begin{equation*}
		d(w_{n_k},w_{n_{k-1}})\ge C_0^{-1}\frac{1}{\sigma}
		\left(1-\frac{n_{k-1}}{n_k}\right)\frac{1}{\sigma^{1/\sigma}b(n_{k})}
		\ge \frac{1}{2}C_0^{-1}\frac{1}{\sigma}(1-\gamma)w_{n_k}.
		\end{equation*}
		Then, for $C_4:=\frac{1}{2}C_0^{-1}\frac{1}{\sigma}(1-\gamma) > 0$, we obtain
		\begin{eqnarray*}
			\frac{ \omega(R_k)}{V(w_{n_k})}\ge
			\frac{C_3}{1+ C_3} \frac{C_4}{1+ C_4} \frac{\omega(w_{n_{k}})}{V(w_{n_k})}.
		\end{eqnarray*}
		Therefore, thanks to hypothesis~(\ref{positive liminf}), we conclude that there exists a constant $ C_ 5 > 0 $ such that, for $ k=1 \mod 3 $ large enough,
		$$  S_{m_k} f(w_{n_{k-1}}) > C_5. $$
		
		We will show in Subsection~\ref{s:Sub3} that $m(f,T)=0$ for $ \xi $ large enough. Let us assume this fact for a moment and argue that the inequality $$f\le u\circ T-u$$ is impossible for every continuous function $u:[0,1] \to \mathbb R$. Suppose the opposite happens. Then, if $ k=1 \mod 3 $ is sufficiently large, we have shown that
		\begin{eqnarray*}
			u(w_{n_{k}}) = u\left(T^{m_k}\left(w_{n_{k-1}}\right)\right)
			&\ge & S_{m_k} f(w_{n_{k-1}}) + u(w_{n_{k-1}})\\
			&>&C_5 + u(w_{n_{k-1}}).
		\end{eqnarray*}
		Since $u$ is continuous at $0$, by letting $k\to +\infty$, we get a contradiction.
		
		\subsection{A condition for  $m(f,T)=0$}\label{s:Sub3}
		
		It remains to argue that, for $ \xi $  large enough, $ m(f,T) = 0 $.
		Since $f(0)=0$ and $\delta_0$ is $T$-invariant, clearly $ m(f,T) \ge \int f d\delta_0= f(0)=0 $.
		If $ \xi $ is sufficiently large, by choosing a suitable constant $ \gamma  \in (0, 1) $ and an appropriate initial point $ w_0 $ close enough to $ 0 $,
		we will show that for each $ x $ there is $n(x)$ such that
		$ S_{n(x)} f(x)\le 0 $. From Birkhoff's ergodic theorem, we thus conclude that $ m(f,T) \le 0 $,
		which completes the proof.
		
		We first choose $ \gamma \in (0, 1) $ satisfying
		\begin{equation}\label{escolha de gama}
		\gamma^{ 1 + 1/\sigma } > \frac{6}{7}.
		\end{equation}
		Note now that, replacing $ w_0 $ by $ w_{n_0} $ with $ n_0 $ large enough, we may assume that the constant $ C_0 $ in Remark~~\ref{Remark consecutive sequences} is as close as we want to $ 1 $.
		Thus, we suppose henceforth that
		\begin{equation}\label{limitando C_0}
		1 < C_0^2 \le \frac{7}{6}\gamma^{1+1/\sigma}.
		\end{equation}
		Furthermore, thanks to~\eqref{gamma mais 1}, if $ n_0 $ is sufficiently large, we may also assume that
		\begin{equation}\label{comportamento para todo k}
		\frac{1}{2} \gamma^{1+1/\sigma}  \le \frac{n_k b(n_k)}{n_{k+1}b(n_{k+1})} \quad \forall \, k \ge 0.
		\end{equation}
		If  $x\in [0,1]\backslash \displaystyle\bigcup_{k=1 \mod 3} I_k$, just take $n(x)=1$, since $f(x)\le 0.$
		Suppose then $x\in I_k$ for some $ k=1 \mod 3$. Define
		$$ p(x) := \min \{ p \ge 1 : T^p(x) \notin I_k \}. $$
		Note that
		$$ S_{p(x)} f(x) \le \# \{j \ge 0 : T^j(x) \in I_k \} \, \omega \Big( \frac{2}{5} \max\{d(w_{n_{k+1}}, w_{n_k}),  d(w_{n_k}, w_{n_{k-1}})\} \Big). $$
		Let us estimate the cardinality in the right term. Denote
		$$ L_k := \Big\lceil \frac{3}{7} C_0 \sigma^{1+1/\sigma}  n_{k} b(n_{k}) d(w_{n_{k}}, w_{n_{k-1}}) \Big\rceil. $$
		From Remark~\ref{Remark consecutive sequences},  we have $ d(w_{n_k}, w_{n_k - L_k}) \ge L_k C_0^{-1}\frac{1}{\sigma^{1+1/\sigma}} \frac{1}{n_k b(n_k)} > \frac{2}{5} d(w_{n_{k}}, w_{n_{k-1}}) $,
		which means that $ w_{n_k - L_k} $ is greater than the right endpoint of $ I_k $.
		Thanks to~\eqref{escolha de gama}, \eqref{limitando C_0} and~\eqref{comportamento para todo k},
		\begin{equation*}
		\frac{3}{7} C_0 \sigma^{1+1/\sigma}  n_{k+1} b(n_{k+1}) d(w_{n_{k+1}}, w_{n_{k}}) \le \frac{3}{7} C_0^2 \frac{n_{k+1}b(n_{k+1})}{n_k b(n_k)} (n_{k+1} - n_k) \le n_{k+1} - n_k,
		\end{equation*}
		so that $ L_{k+1}  \le n_{k+1} - n_k $.
		Hence, a similar reasoning shows that  $ w_{n_k + L_{k+1}} $ is smaller than the left endpoint of $ I_k $.
		Therefore, by the monotonicity of $ T $, we obtain
		\begin{align*}
		\# \{j : T^j(x) \in I_k \}
		& \le (L_k - 1) + (L_{k+1} - 1) \\
		& \le \frac{3}{7} C_0  \sigma^{1+1/\sigma}  n_{k+1} b(n_{k+1}) d(w_{n_{k+1}}, w_{n_{k-1}}).
		\end{align*}
		We have shown that
		\begin{equation}\label{majoracao_p}
		S_{p(x)} f(x) \le \frac{3}{7} C_0\sigma^{1+1/\sigma}  n_{k+1} b(n_{k+1}) d(w_{n_{k+1}}, w_{n_{k-1}}) \, \omega \big( d(w_{n_{k+1}}, w_{n_{k-1}}) \big).
		\end{equation}
		Now, for $ y \in \big[ w_{n_{k+1}} + R_{k+1}, \frac{1}{5} (3 w_{n_k} + 2 w_{n_{k+1}})\big] $, denote
		$$ q(y) := \min \{ q \ge 1 : d(T^q(y), w_{n_{k+1}}) < R_{k+1} \}. $$
		Clearly,
		$$ S_{q(y)} f(y) \le - \xi \, \# \big\{j \ge 0 : R_{k+1} \le d(T^j(y), w_{n_{k+1}}) \le \frac{1}{3} d(w_{n_{k+1}}, w_{n_k}) \big\} \, \omega(R_{k+1}). $$
		Thanks to Lemma~\ref{minoracao cardinalidade}, we obtain that
		\begin{equation}\label{majoracao_q}
		S_{q(y)} f(y) \le - \xi \, C_1 n_{k} b(n_{k}) d(w_{n_k}, w_{n_{k+1}}) \omega (R_{k+1}).
		\end{equation}
		We claim that, whenever $ \xi $ is sufficiently large, for $ n(x) := p(x) + q(T^{p(x)} (x)) $ one has $ S_{n(x)}f(x) \le 0 $ .
		Thanks to~\eqref{majoracao_p} and~\eqref{majoracao_q}, it is enough to prove that
		$$ \sup_k \frac{n_{k+1} b(n_{k+1}) d(w_{n_{k+1}}, w_{n_{k-1}}) \omega \big( d(w_{n_{k+1}}, w_{n_{k-1}}) \big)}{ n_{k} b(n_{k}) d(w_{n_k}, w_{n_{k+1}}) \omega(R_{k+1})} < \infty. $$
		Recalling the asymptotic equivalence~\eqref{gamma mais 1}, we just have to show that both suprema
		$$  \sup_k  \frac{d(w_{n_{k+1}}, w_{n_{k-1}})}{ d(w_{n_k}, w_{n_{k+1}})} \quad \text{and} \quad
		\sup_k  \frac{\omega \big( d(w_{n_{k+1}}, w_{n_{k-1}}) \big)}{ \omega (R_{k+1})} $$
		are finite. With respect to the first one, from~(\ref{minoracao e majoracao}) it is immediate that
		\begin{eqnarray}\label{majoracao quociente distancias}
		\nonumber \frac{d(w_{n_{k}}, w_{n_{k-1}})}{d(w_{n_{k+1}}, w_{n_k})} & \le & \frac{C_0 (n_{k}  - n_{k-1}) 1/\big[\sigma^{1+1/\sigma} n_{k-1} b(n_{k-1})\big]}
		{C_0^{-1} (n_{k+1}  - n_{k}) 1/\big[\sigma^{1+1/\sigma} n_{k+1} b(n_{k+1})\big]}\\
		&=& C_0^2 \, \frac{1 - \frac{n_{k-1}}{n_k}}{\frac{n_{k+1}}{n_k }- 1} \, \frac{n_{k+1}b(n_{k+1})}{n_{k}b(n_{k})} \, \frac{n_{k}b(n_{k})}{n_{k-1}b(n_{k-1})},
		\end{eqnarray}
		which ensures $ \frac{d(w_{n_{k+1}}, w_{n_{k-1}})}{ d(w_{n_k}, w_{n_{k+1}})} = 1 + \frac{d(w_{n_{k}}, w_{n_{k-1}})}{d(w_{n_{k+1}}, w_{n_{k}})} $ is bounded from above.	
		With respect to the second one, note first that, thanks to~\eqref{majoracao quociente distancias},
		\begin{equation*}\label{majoracao quociente distancia e Rk}
		\frac{d(w_{n_{k+1}}, w_{n_{k-1}})}{R_{k+1}} = 3C_0^3 \, \frac{n_{k+1} b(n_{k+1})}{n_{k}b(n_{k})}  \,  \frac{d(w_{n_{k+1}}, w_{n_{k-1}})}{ d(w_{n_k}, w_{n_{k+1}})}
		\end{equation*}
		is bounded from above. Hence, there exists a positive constant $ C_6 $ such that
		$ d(w_{n_{k+1}}, w_{n_{k-1}}) \le C_6 R_{k+1} $. By the monotonicity of $ \omega $ and Lemma~\ref{modulo constante}, we obtain
		$$	\frac{\omega\big( d(w_{n_{k+1}}, w_{n_{k-1}}) \big)}{ \omega (R_{k+1})} \le
		C_6+1<\infty.$$
		The proof is complete.
		
		\medskip
		
		\paragraph{Comments on the proof of Theorem~\ref{t:Theorem 1} for $x\mapsto x(1+V(x))$.}
		We consider now a subsequence $(w_{n_k})$ that fulfills  $ w_{n_{k-1}}= T^{n_{k}-n_{k-1}}(w_{n_k})$, where $T(x)= x(1+V(x))$ in a neighborhood of $0$.
		Note that orbits are moving monotonically away from the origin, that is, they are moving to the right instead of to the left as in the previous case.
		This merely produces a, let us say, \emph{reflexive effect} on our arguments, exchanging the roles of indices $  k=1 \mod 3 $ and $  k=2 \mod 3 $.
		In practical terms, we define $\Phi$ for this case as
		$$\Phi(x):= \left \{
		\begin{array}{cc}
		-\xi \varphi_{k}(x),& x\in I_k, \,\, k=1 \mod 3 \\
		
		\varphi_{k}(x),& x\in I_k, \,\, k=2 \mod 3 \\
		
		0,&  \textrm{otherwise}.
		\end{array}
		\right.$$
		Introducing $ f $ as in~(\ref{definicao funcao}) and supposing by a moment that  $m(f,T)=0$, we apply the same strategy to show that $f$ does not admit continuous sub-action.
		In fact, for $k=2\mod 3$ sufficiently large, using~\eqref{segundo minorante alternativo} one estimates  the number of iterates that remain in the interval $[\frac{1}{3}(2 w_{n_k}+ w_{n_{k+1}}), w_{n_k}-R_{k+1}]$  to conclude that
		$ S_{m_{k+1}} f(w_{n_{k+1}}) $ is bounded from below by a positive constant and thus to reach a contradiction.
		In order to show that, for the same choice of parameters~\eqref{escolha de gama}, \eqref{limitando C_0}, and~\eqref{comportamento para todo k}, $m(f,T)=0$ whenever $ \xi $ is sufficiently large, suitable adjustments
		are required to obtain that for $x\in I_k$ with $k= 2\mod 3$, there is $n(x)$ such that $S_{n(x)} f(x)\le 0.$ Similarly to the previous case, the key observation is that such a Birkhoff sum may be bounded from above
		by the difference
		of two terms, the first one takes into account the iterates that remain in $I_k$, the second one considers iterates that remain in $[\frac{1}{3}(2 w_{n_{k-1}}+w_{n_k}), w_{n_{k-1}}-R_k]$, and
		their ratio is uniformly bounded.
		
		\section*{Appendix: On the existence of  sub-actions}\label{existence}
		Since the analysis of the existence of sub-actions is a global issue, we fix a particular class of  dynamics  with intermittent behavior. Our working class of maps  with two branches  provides an example of situation in which one can guarantee the existence of sub-actions for potentials with various  moduli of continuity, highlighting clearly the associate  regularity of these sub-actions.
		Similar arguments are feasible  for intermittent dynamics with more inverse branches.
		
		At the best of our knowledge, there  are no previous works at such a level of  generality about  the regularity of potentials and sub-actions.
		\medskip

		Throughout this section we consider
		a class $ \mathscr J$ of one-dimensional maps, so that  each $T\in \mathscr J$ is a piecewise two to one interval map defined on $([0,1],d)$ with
		discontinuity $c\in(0,1)$ such that
		$ \displaystyle\lim_{x\to c^-} T(x)=~1 \textrm{ and } \lim_{x\to c^+} T(x)=~0.$ Moreover, $T$ takes the form $T(x):=x(1+ V(x))$  on $[0,c], $
		where for  some $\sigma>0$,  the continuous and increasing function $V:[0,+\infty)\to [0,1)$ is regularly varying with index $\sigma$ (recall \eqref{regularly variation}). Finally, we  assume that there is $\lambda>1$ such that  for all $x,y\in (c,1]$,
		$d(T(x),T(y))\ge \lambda d(x,y).$
		
		As in $\S$~\ref{statements}, $\mathcal M$ denotes the set of continuous, non-decreasing, concave modulus of continuity.
		For a given function $V$ as above, we consider  an appropriate $\omega \in \mathcal M$ satisfying the following assumption:
		\begin{quote}
			[A]	There exist constants  $\gamma>0$, $\xi_0>1$ and  $\eta_0\in (0,1)$ such that
			\begin{equation}\label{minoracao de omega}
			\frac{\omega(\xi h)}{V(\xi h)}\ge \xi^\gamma \frac{\omega(h)}{V(h)}, \qquad \forall\, h\in (0,\eta_0), \, \forall \, \xi\in (1,\xi_0].
			\end{equation}
			
		\end{quote}
		One can easily verify that, for $V$ and  $\omega$ fulfilling    \eqref{minoracao de omega},
		\begin{equation}\label{limite omega V}
		\lim_{h\to 0}\frac{\omega(h)}{V(h)}=0.
		\end{equation}
		The converse statement is not satisfied in general, see Remark~\ref{remark gama zero}.

		From Assumption A, we define a  modulus of continuity $\Omega\in \mathcal M$  so  that  potentials with modulus of continuity $\omega$  admit  sub-actions with modulus of continuity $\Omega$. Before we state this result, we first provide examples of  maps in $\mathscr J$ for which   condition \eqref{minoracao de omega} holds.
		
		\subsubsection*{Examples}\label{exemplos de existencia}
		A prototypical example  in $\mathscr J$ is the  Manneville-Pomeau interval map defined for some  $s\in(0,1)$ as $  T_s(x):= x(1 + x^s) \mod 1.$
		Consider the class of modulus of continuity $\omega_{\alpha,\beta}$ as in \eqref{modulus alpha beta}. For $s<\alpha<1$,  condition \eqref{minoracao de omega} follows immediately with $\gamma=\alpha -s$: for $h$ sufficiently small,
		\begin{equation*}
		\frac{\omega_{\alpha,\beta}(\xi h)}{(\xi h)^s}
		\ge
		\xi^{\alpha-s} \frac{h^\alpha(-\log h)^{-\beta}}{h^s}= \xi^{\alpha-s} \frac{\omega_{\alpha,\beta}(h)}{h^s}.
		\end{equation*}
		Another interesting   family of interval maps in $\mathscr J$  is given by
		$ H_\rho:[0,1]\to [0,1]$, for  $\rho\in (0,1]$, defined as
		\begin{equation*}
		H_\rho(x) =\left\{
		\begin{array}{ll}
		\frac{x}{(1-x^{\rho})^{1/\rho}}  & \textrm{ if } 0\le x\le 2^{-1/\rho},\\
		\frac{2^{1/\rho}x-1}{2^{1/\rho}-1} &\textrm{ if }  2^{-1/\rho}<x\le 1.
		\end{array}
		\right.
		\end{equation*}
		The function  $V(h)= \frac{1}{(1-h^\rho)^{1/\rho}}-1$ is continuous, increasing, regularly varying with index~$\rho$. For $\rho<\alpha<1$, we have that $\omega_{\alpha,\beta}$ and $V$ satisfy  condition \eqref{minoracao de omega}, since
		\begin{align*}
		\frac{\omega_{\alpha,\beta}(\xi h)}{V(\xi h)} \frac{V( h)}{\omega_{\alpha,\beta}(h)}&
		=\xi^\alpha\left(\frac{\log(\xi h)}{\log h}\right)^{-\beta}				                 		\frac{V(h)}{V(\xi h)}		
		\end{align*}
		implies that
		$\displaystyle
		\lim_{h\to 0}\frac{\omega_{\alpha,\beta}(\xi h)}{V(\xi h)} \frac{V( h)}{\omega_{\alpha,\beta}(h)}=
		\xi^\alpha  \lim_{h\to 0}\frac{V( h)}{V(\xi h)}
		=\xi^{\alpha-\rho}.
		$
		As another example, following \cite{Hol05}, consider a family  defined for $0<\tau<1$ and $\theta>0$ as
		\begin{equation*}
		T_{\tau,\theta}(x) =\left\{
		\begin{array}{ll}
		x+ \, \frac{2^\tau}{(\log 2)^{\theta+1}}x^{1+\tau} \vert \log x\vert^{\theta+1} & \textrm{ if } 0\le x\le 1/2, \\
		2x-1 &\textrm{ if } 1/2 < x \le 1.
		\end{array}
		\right.
		\end{equation*}
		In this case, the function $V_{\tau,\theta}(h)=\frac{2^\tau}{(\log 2)^{\theta+1}}h^{\tau}\vert \log h\vert^{\theta+1}$ is  regularly varying with index~$\tau$.
		Condition~\eqref{minoracao de omega} is satisfied, for instance, with the modulus of continuity  $\omega_k(h)=h\left(\log\left(\frac{1}{h^k}\right)+1\right)$  for $k\ge 1$ and  $ h $ sufficiently small. Indeed, one has
		\begin{align*}
		\frac{\omega_k(\xi h)}{V_{\tau,\theta}(\xi h)} \frac{V_{\tau,\theta}(h)}{\omega_k(h)} &
		=\xi^{1-\tau}\left\vert\frac{\log h}{\log(\xi h)}\right\vert^{\theta+1}
		\frac{1-k\log(\xi h)}{1-k\log h},			                		 \end{align*}
		so that
		$\displaystyle
		\lim_{h\to 0}\frac{\omega_k(\xi h)}{V_{\tau,\theta}(\xi h)} \frac{V_{\tau,\theta}
			(h)}{\omega_k(h)}=
		\xi^{1-\tau}\lim_{h\to 0}
		\frac{1-k\log(\xi h)}{1-k\log h}
		=
		\xi^{1-\tau}.
		$
		
		\begin{remark}\label{remark gama zero}
			[Condition~\eqref{minoracao de omega} is more restricted than \eqref{limite omega V}.]
			For $\theta>0$ and $k\ge 1,$ consider $T_{1,\theta}$ and $\omega_k$ as above.
			It is easy to see that
			\begin{equation*}
			\frac{\omega_k(h)}{V_{1,\theta}(h)}\to 0\quad  \textrm{ as } h\to 0.
			\end{equation*}
			However, from
			$\displaystyle
			\frac{\omega_k(\xi h)}{V_{1,\theta}(\xi h)} \frac{V_{1,\theta}(h)}{\omega_k(h)}   =
			\left\vert\frac{\log h}{\log(\xi h)}\right\vert^{\theta+1}
			\frac{1-k\log(\xi h)}{1-k\log h}, 		                		
			$
			we get
			\begin{align*}
			\lim_{h\to 0}\frac{\omega_k(\xi h)}{V_{1,\theta}(\xi h)} \frac{V_{1,\theta}(h)}{\omega_k(h)}=1.
			\end{align*}
			Hence,  property \eqref{limite omega V} is satisfied, however~\eqref{minoracao de omega} fails.
		\end{remark}
		
		\subsubsection*{Defining a continuous increasing concave modulus of continuity}
		For $V$ and $\omega$ fulfilling \eqref{minoracao de omega}, let
		$\vartheta_0:[0,\infty)\to [0,\infty)$ be the continuous function defined as
		\begin{equation}
		\vartheta_0(x):=\left\{
		\begin{array}{ll}
		\frac{\omega(x)}{V(x)},& x>0,  \\
		0,& x= 0,
		\end{array}
		\right.
		\end{equation}
		and let $ \vartheta_1:[0,\infty)\to [0,\infty) $ be  the continuous increasing function given as
		\begin{equation}
		\vartheta_1(x)=\left\{
		\begin{array}{ll}
		\displaystyle \max_{0\le y\le x} \vartheta_0(y),& 0\le x\le 1,\\
		\displaystyle\max_{[0,1]}\vartheta_0, & x\ge 1,
		\end{array}
		\right.
		\end{equation}
		Denote then $\vartheta_1^*$  the \emph{concave conjugate Legendre transform} of $\vartheta_1$, defined  as
		\begin{equation}
		\vartheta_1^*(x)= \min_{y\in[0,\infty)}[xy-\vartheta_1(y)], \quad \forall\, x\ge 0.
		\end{equation}
		By the very definition, $\vartheta_1^*$ is concave, increasing and continuous on $(0,\infty).$ To see that $\vartheta^*$ is continuous at $0$, note that
		$\vartheta^*_1(0)=-\max_{[0,1]}\vartheta_0$ and
		$\vartheta_1^*(0)\le \vartheta_1^*(\epsilon)\le \epsilon -\vartheta_1(1)= \epsilon+\vartheta_1^*(0).$
		For the continuous concave increasing function
		\begin{equation}
		\vartheta_2(x)=\min\{\vartheta_1^*(x), \vartheta_1^*(1)\},
		\end{equation}
		a similar reasoning shows that
		its  \emph{concave conjugate Legendre transform},
		\begin{equation}\label{concave biconjugate}
		\vartheta_2^{*}(x)= \min_{y\in [0,\infty)}[xy-\vartheta_2(y)], \quad \forall\,  x\ge 0,
		\end{equation}
		is also a continuous concave increasing function. Moreover
		$\vartheta_0(x)\le \vartheta_1(x)\le \vartheta_2^*(x)$ for all $x\in [0,1].$
		Actually, $\vartheta_2^{*}$ is the \emph{smallest} concave function that lies above $\vartheta_1$ on~$[0,1].$
		Note that $\vartheta_2^*(0)=-\vartheta_1^*(1).$
		
		We have obtained a function $\Omega:= \vartheta_2^{*}+\vartheta_1^*(1)$ that belongs to $\mathcal M$.
		\begin{theorem}\label{theorem existence} Let $T:[0,1]\to [0,1]$ be a map in $\mathscr J$ with discontinuity $c\in (0,1)$ such that $T(x)=x(1+V(x))$  for all $x\in [0,c]$, where  $V$ is regularly varying at~$0$. Let $\omega $ be a modulus of continuity in $\mathcal M$ for which  Assumption A holds. Then, every $f\in \mathscr C^{\omega}([0,1])$ admits continuous sub-actions in $ \mathscr C^{\Omega}([0,1])$, where $\Omega$ is defined by the process $(23)$-$(27).$
		\end{theorem}
		\subsubsection*{Proof of the  theorem}
		In the following results we will assume the hypotheses of Theorem~\ref{theorem existence}. In particular,  we keep in mind all the  constants of Assumption A.
		\begin{lemma}\label{desigualdade distancia}
			There are constants $\varrho_T>0$ and $C_7\in (0,\min\{\xi_0,\eta_0^{-1}\}-1]$ such that for all $x,y\in [0,1]$, with $d(x,y)<\varrho_T$, we have
			\begin{equation}
			d(T(x),T(y))\ge d(x,y) \big(1+C_7 V(d(x,y))\big).
			\end{equation}
		\end{lemma}
		\begin{proof}
			Let $x,y\in [0,c]$ with $x<y$. Since $V$ and $T$ are increasing, note that
			\begin{align*}
			d(T(x),T(y))=  d(x,y)+ d(x,y)\,V(y)+x\,(V(y)-V(x))
			\ge d(x,y)\,\big(1+ V(d(x,y))\big).
			\end{align*}
			Consider now $x,y\in (c,1]$. Since $V([0,1])\subset [0,1)$,  we clearly have
			\begin{align*}
			d(T(x),T(y))\ge \lambda\, d(x,y)\, \ge d(x,y)\, \big(1+(\lambda-1)\, V(d(x,y))\big).
			\end{align*}
			Fix $\varrho>0$ such that, for $x\in [c-\varrho/2,c)$ and $y\in (c,c+\varrho/2]$ it follows that $d(T(x),T(y))\ge 1/2.$
			We choose $\varrho_T\in (0,\varrho)$ such that $V\big(\frac{1}{2}h\big)\ge \frac{1}{2^{\sigma+1}} V(h)$
			for all  $h\in [0,\varrho_T]$.
			Then  for
			$c-\varrho_T/2\le x<c<y\le c+\varrho_T/2$, 	\begin{align*}
			d(T(x),T(y))&\ge  1-d(T(x),T(y))= \lim_{t\to c^-} d(T(t),T(x))+\lim_{t\to c^{+}}d(T(y),T(t))\\
			&\ge \lim_{t\to c^{-}}  d(t,x)\,\big(1+ V(d(t,x))\big)+ \lim_{t\to c^+}d(y,t)\,\big(1+ (\lambda-1)\,V(d(y,t))\big)\\
			&= d(x,y)+ d(c,x)\,V(d(c,x))+ (\lambda-1)\, d(y,c)\, V(d(y,c)).
			\end{align*}
			Suppose that $d(c,x)\ge d(y,c)$, then  $2\, d(c,x)\ge d(x,y)$ and
			\begin{align*}
			d(T(x),T(y))&\ge
			d(x,y)+\frac{1}{2}\, d(x,y) \,V\Big(\frac{1}{2} d(x,y)\Big)
			\ge  d(x,y)+\frac{1}{2^{\sigma+2}}\, d(x,y)\, V(d(x,y)).
			\end{align*}
			Similarly, if $d(c,y)\ge d(x,c)$, then $2\, d(c,y)\ge  d(x,y)$ and
			\begin{align*}
			d(T(x),T(y))
			\ge d(x,y)+\frac{(\lambda-1)}{2^{\sigma+2}} \,d(x,y)\, V(d(x,y)).
			\end{align*}
			Take $C_7:=\displaystyle \min\left\{\frac{1}{2^{\sigma+2}} ,\frac{\lambda-1}{2^{\sigma+2}},\xi_0-1, \frac{1}{\eta_0}-1\right\}.$
		\end{proof}
		\begin{proposition}\label{lemma concave conjugate}
			There are constants $\varrho_{T,\omega}>0$  and
			$C_8>0$ such that, given  a sequence $\{x_k\}_{k\ge 0} $ in $[0,1]$,  with $T(x_{k+1})=x_k$ for $k\ge 0$, and   a point $y_0\in [0,1] $ with  $d(x_0,y_0)< \varrho_{T,\omega}$, there is $\{y_k\}_{k\ge 1}\subset [0,1],$ with $T(y_{k+1})=y_k$ for $k\ge 0$, satisfying
			\begin{equation}\label{equation concave conjugate}
			\Omega\big(d(x_k,y_k)\big)+ C_8\, \sum_{j=1}^k \omega\big(d(x_j,y_j)\big) \le \Omega\big(d(x_0,y_0)\big) \qquad \forall\, k\ge 1.
			\end{equation}
			
		\end{proposition}
		\begin{proof} Let $\varrho_{T,\omega}=\min\{\varrho_T,\eta_0\}$, where $\varrho_T$ is as in the statement of  Lemma \ref{desigualdade distancia}. For   $x_0, x_1,  y_0\in [0,1]$  with  $T(x_1)=x_0$ and  $d(x_0,y_0)<\varrho_{T,\omega}$, we can choose  $y_1\in T^{-1}(y_0)$ with
			$d(x_1,y_1)\le d(x_0,y_0)<\varrho_{T,\omega}$. Then from
			Lemma~\ref{desigualdade distancia},
			\begin{equation*}
			d(x_0,y_0)=d(T(x_1), T(y_1))\ge d(x_1,y_1)\, \big(1+ C_7\, V(d(x_1,y_1))\big).
			\end{equation*}
			Since $\Omega$ is increasing, we have
			$
			\Omega\big(d(x_0,y_0)\big) \ge \Omega\big(d(x_1,y_1)\,\big (1+ C_7\, V(d(x_1,y_1)\big)\big).
			$  	
			For $h=d(x_1,y_1)$, we can  write
			$$
			\Omega\big(h (1+C_7\, V(h))\big)
			=
			\Omega\big((1-V(h))\,h+V(h)\,(1+C_7)\, h)\big).
			$$
			As  $\Omega=\vartheta_2^*+\vartheta_1^*(1)$ is concave, we see that
			\begin{align*}
			\Omega\big(h\,(1+C_7\, V(h))\big)&\ge
			(1-V(h))\,\Omega(h)+V(h)\,\Omega\big((1+C_7)\,h\big)\\
			& =
			\Omega(h)+ V(h)\,\Big (\vartheta_2^*((1+C_7)\,h)-\vartheta_2^*(h)\Big).
			\end{align*}
			Recalling that $\vartheta_2^*\ge \vartheta_0$, we have
			\begin{align*}
			\Omega\big(h\,(1+C_7\, V(h))\big)&\ge
			\Omega(h)+ V(h)\,\vartheta_2^*(h)\,\Big(\frac{\vartheta_2^*\big((1+C_7)\,h\big)}{\vartheta_2^*(h)}-1\Big)\\
			&\ge
			\Omega(h)+ \omega(h)\,\Big(\frac{\vartheta_2^*\big((1+C_7)\,h\big)}{\vartheta_2^*(h)}-1\Big).
			\end{align*}
			We claim that $\frac{\vartheta_2^*((1+C_7)\, h)}{\vartheta_2^*(h)}\ge (1+C_7)^\gamma$. As a matter of fact, following Assumption~A,  for $1+C_7\le\xi_0$, since  $h=d(x_1,y_1)< \varrho_{T,\omega}\le \eta_0,$
			\begin{equation*}
			\displaystyle  \frac{\vartheta_0((1+C_7)\,h)}{\vartheta_0(h)}\ge (1+C_7)^\gamma,\quad \textrm{ and thus } 	
			\displaystyle \frac{\vartheta_1((1+C_7)\,h)}{\vartheta_1(h)}\ge~
			(1+C_7)^\gamma.
			\end{equation*}
			Write  $\xi= 1+C_7$ and recall that the transform Legendre is order reversing, then
			$$ 	\vartheta_2\Big(\frac{h}{\xi}\Big)=\vartheta_1^*\Big(\frac{h}{\xi}\Big)=(\vartheta_1(\xi\,h))^*\le~(\xi^\gamma \vartheta_1(h))^*= \xi^\gamma \vartheta_1^*\Big(\frac{h}{\xi}\Big)= \xi^\gamma\vartheta_2\Big(\frac{h}{\xi^\gamma}\Big). $$
			Applying again the concave conjugate,  we get          	  $$\vartheta_2^*(\xi h)=\Big(\vartheta_2\Big(\frac{h}{\xi}\Big)\Big)^*\ge \Big(\xi^{\gamma}\vartheta_2\Big(\frac{h}{\xi^\gamma}\Big)\Big)^*= \xi^\gamma \vartheta_2^*(h).$$
			Therefore, for $C_8:= (1+C_7)^\gamma-1$,  we have shown that,  for $x_0,x_1$, $y_0\in [0,1]$ with $T(x_1)=x_0$  and  $d(x_0,y_0)<\varrho_{T,\omega}$,  there is $y_1\in T^{-1}(y_0)$, with $d(x_1,y_1)\le d(x_0,y_0)< \varrho_{T,\omega}$, such that
			$$\Omega\left(d(x_0,y_0)\right)\ge \Omega(d(x_1,y_1))+ C_8\,\omega(d(x_1,y_1)).$$
			Inequality \eqref{equation concave conjugate} follows straightforward from the above inequality.
		\end{proof}
		
		For $\omega\in \mathcal M$ and  $\varphi \in \mathscr C^\omega([0,1])$, we  denote
		$$\vert \varphi \vert_\omega =\sup_{x\neq y}\frac{\vert\varphi(x)-\varphi(y)\vert}{\omega(d(x,y))}.$$
		
		\begin{lemma}\label{suma de Birkhoff}
			Let $ g_k(x):= \sup_{T^k(y)=x} S_k\big(f-m(f,T)\big)(y)$, for  $k\ge 1$. Then,
			there is $L=L(\varrho_{T,\omega})> 0$ such that for every $k\ge 1$,
			\begin{align*}
			\vert g_k(x)-g_k(y)\vert & \le L\, C_8^{-1}\,\vert f\vert_\omega\, \Omega(d(x,y)), \qquad & \forall \, x,y\in [0,1] \textrm{ and } \\
			\vert g_k(x)\vert &\le 2\,L\,C_8^{-1}\,\vert f\vert_\omega\, \Omega(1),\qquad  & \forall \, x\in [0,1],
			\end{align*}
			where $\varrho_{T,\omega} $ and $ C_8 $ are as in the statement of Proposition~\ref{lemma concave conjugate}.
		\end{lemma}
		
		\begin{proof}
			Without loss of generality, we suppose that $m(f,T)=0$. Let $x_0,y_0\in [0,1]$ be  such that $d(x_0,y_0)<\varrho_{T,\omega}$. Fix $k\ge 1$ and assume that $g_k(x_0)\ge g_{k}(y_0)$.
			Given $ \epsilon > 0 $, there exists $x_k\in~T^{-k}(x_0)$ with 	$g_k(x_0) - \epsilon < S_k f(x_k)$. We apply the previous proposition and consider $y_k\in T^{-k}(y_0)$ so that
			$$
			\sum_{j=0}^{k-1}\omega\big(d\big(T^{j}(x_k), T^{j}(y_k)\big)\big) \le C_8^{-1}\Big(
			\Omega\big(d(x_0,y_0)\big)- \Omega\big(d(x_k,y_k)\big)\Big)\le C_8^{-1}\Omega\big(d(x_0,y_0)\big).
			$$
			Thus,
			\begin{align*}
			\vert g_k(x_0)-g_k(y_0)\vert - \epsilon &<  S_k f(x_k)-S_k f(y_k)\\
			& \le \vert f\vert_\omega\sum_{j=0}^{k-1}\omega \big(d\big(T^{j}(x_k), T^{j}(y_k)\big)\big) \le C_8^{-1}\, \vert  f \vert_\omega \,\Omega(d(x_0,y_0)).
			\end{align*}
			Therefore, as $ \epsilon > 0 $ is arbitrary, if $d(x_0,y_0)<\varrho_{T,\omega}$ and $k\ge 1$,
			$$ \vert g_k(x_0)-g_k(y_0)\vert  \le C_8^{-1}\, \vert  f\vert_\omega \,\Omega(d(x_0,y_0)). $$
			For $z\in [0,1]$, define $I_z= (z-\varrho_{T,\omega}/2,z+\varrho_{T,\omega}/2)\cap [0,1]$.
			There are finitely many points  $z_i\in [0,1]$, $1\le i\le L-1$, which are assumed ordered, such that $\{I_{z_i}\}_{i=1}^{L-1}$ is an open cover of $[0,1]$.
			Hence, given $x + \varrho_{T,\omega} \le y$ in $[0,1]$, consider indexes  $i_x<i_y$ for which
			$x\in I_{z_{i_x}} $ and $ y\in I_{z_{i_y}}$. Note that, as  $\Omega$ is increasing, the above local property provides
			\begin{align*}
			\vert g_k(x)-g_k(y)\vert \le & \vert g_k(x)-g_k(z_{i_x})\vert + \sum_{i_x\le i< i_y} \vert g_k(z_{i})-g_k(z_{i+1})\vert+ \vert g_k(z_{i_y})- g_k(y)\vert\\
			\le &  L \,C_8^{-1}\,\vert f\vert_\omega \,\Omega\big(d(x,y)\big).
			\end{align*}
			We have shown that  the family $\{g_k\}_{k\ge 1} $ is equicontinuous. To obtain uniform boundness, denote  $C_9=L\,C_8^{-1}\,\vert f\vert_\omega\, \Omega(1)$. By contradiction, suppose that for some  $\tilde x \in[0,1]$ and $k_0\ge 1$, one has
			$\vert g_{k_0}(\tilde x ) \vert > 2C_9 $. By the previous discussion, we would have
			$
			\vert g_{k_0}(\tilde x)-g_{k_0}(x)\vert \le C_9
			$
			for all $x\in [0,1], $
			so that $\vert g_{k_0}\vert> C_9 $ everywhere.
			Then  there would be a sequence $(\tilde x_\ell)_{\ell\ge 1}$  such that $T^{\ell k_0}(\tilde x_\ell)=\tilde x$ and
			$S_{\ell k_0} f(\tilde x_\ell)>~\ell\,C_9$, hence
			\begin{equation*}
			\frac{1}{\ell\, k_0} S_{\ell k_0} f(\tilde x_\ell) >\frac{C_9}{k_0}>0.
			\end{equation*}
			This contradicts the fact that $m(f,T)=0.$ Indeed, it is easy to see that the Borel probabilities
			$ \nu_\ell = \frac{1}{\ell k_0} \big( \delta_{\tilde x_\ell} + \delta_{T(\tilde x_\ell)} + \ldots + \delta_{T^{\ell k_0 - 1}(\tilde x_\ell)} \big) $ have, with respect
			to the weak-star topology, $T$-invariant measures as accumulation probabilities as $ \ell \to \infty $. Hence, if $ \nu_\infty $ is any one of these accumulation probabilities, then
			$$ m(f, T) \ge \int f \, d\nu_\infty = \lim_{j \to \infty} \frac{1}{\ell_j k_0} S_{\ell_j k_0} f(\tilde x_{k_{\ell_j}}) \ge \frac{C_9}{k_0}. $$
		\end{proof}
		
		\begin{proof}[Proof of Theorem~\ref{theorem existence}]
			Following \cite[Proposition~11]{CLT01}, denote $ g_0 \equiv 0 $ and define, for every $ x\in[0,1] $,
			\begin{equation*}
			U_f(x):= \sup_{k\ge 0}g_k(x) = \sup \Big\{ S_k \big(f - m(f, T)\big) : k \ge 0 \text{ and } T^k(y) = x \Big\}.
			\end{equation*}
			Thanks to Lemma~\ref{suma de Birkhoff}, $ U_f $ is a well-defined real function and actually $U_f \in \mathscr C^{\Omega}([0,1])$.
			Furthermore, it follows from definition that the inequality $ U_f \circ T \ge U_f + f - m(f,T) $ holds and therefore $ U_f $ is a sub-action.
		\end{proof}
		
		\acknowledgment{We are indebted to J. T. A. Gomes for his attentive reading of this appendix.}

	    \bibliographystyle{amsalpha}
	    \bibliography{References}

\end{document}